\documentclass{amsart}

\usepackage{amsmath}
\usepackage{amsthm}
\usepackage{amssymb}
\usepackage{enumitem}
\usepackage{url}
\usepackage[colorlinks,linkcolor=blue,anchorcolor=blue,citecolor=blue,backref=page]{hyperref}
\usepackage{color}
\usepackage[norefs,nocites]{refcheck}

\DeclareMathOperator{\ab}{ab}
\DeclareMathOperator{\z}{z}
\DeclareMathOperator{\rad}{rad}
\DeclareMathOperator{\Res}{Res}

\newcommand*{\QEDB}{\hfill\ensuremath{\square}}%

\newtheorem{theorem}{Theorem}[section]
\newtheorem{corollary}[theorem]{Corollary}
\newtheorem{lemma}[theorem]{Lemma}
\newtheorem{conjecture}[theorem]{Conjecture}

\newcommand{\Address}{{
\bigskip
\footnotesize
\textsc{School of Mathematics and Statistics, University of New South Wales, Sydney NSW 2052, Australia}\par\nopagebreak
\textit{E-mail address:} \texttt{marley.young@student.unsw.edu.au}
}}

\title{On Multiplicative Independence of Rational Function Iterates}
\author{Marley Young}
\date{}

\begin{document}

\subjclass[2010]{11R18, 39B12, 12E99, 37F10}
\keywords{Iteration, multiplicative dependence, rational function}

\begin{abstract}
We give lower bounds for the degree of multiplicative combinations of iterates of rational functions (with certain exceptions) over a general field, establishing the multiplicative independence of said iterates. This leads to a generalisation of Gao's method for constructing elements in the finite field $\mathbb{F}_{q^n}$ whose orders are larger than any polynomial in $n$ when $n$ becomes large. Additionally, we discuss the finiteness of polynomials which translate a given finite set of polynomials to become multiplicatively dependent.
\end{abstract}

\maketitle

\section{Introduction and Main Results}

We say that $n$ non-zero elements $a_1,\ldots,a_n$ of a ring are multiplicatively independent if, for integers $k_1,\ldots,k_n$, we have that $a_1^{k_1}\ldots a_n^{k_n}=1$ if and only if $k_1=\ldots=k_n=0$. Otherwise we say they are multiplicatively dependent. Multiplicative independence, especially of values of polynomials and rational functions, is being increasingly studied. In \cite{BMZ}, Bombieri, Masser and Zannier initiate study of the intersection of algebraic curves with proper algebraic subgroups of the multiplicative group $\mathbb{G}_m^n$. It turns out (see \cite[Corollary~3.2.15]{BG}) that each such subgroup of $\mathbb{G}_m^n$ is defined by finitely many equations of the form $X_1^{k_1}\ldots X_n^{k_n} = 1$, where $k_1,\ldots,k_n$ are integers, not all zero. As such, \cite{BMZ}, which leads into the area of ``unlikely intersections'', really concerns the multiplicative dependence of points on curves. \par
More recently, we see multiplicative independence being studied in the context of arithmetic dynamics. In \cite{OSSZ}, it is shown that under fairly natural conditions on rational functions $f_1,\ldots,f_s$ over a number field $\mathbb{K}$, the values $f_1(\alpha),\ldots,f_s(\alpha)$ are multiplicatively independent for all but finitely many $\alpha \in \mathbb{K}^{\ab}$, where $\mathbb{K}^{\ab}$ is the maximal abelian extension of $\mathbb{K}$. This leads to results on multiplicative dependence in the orbits of a univariate polynomial dynamical system. \par 
Clearly, to study the multiplicative independence of elements in the orbits of polynomials or rational functions, it is necessary to know when the given functions are multiplicatively dependent, as in this case all their values must be multiplicatively dependent. We study this problem in the context of iterates of rational functions over a field.

Throughout the paper, $\mathbb{F}$ will denote a field of characteristic $p$ (zero or prime), and $f \in \mathbb{F}(X)$ a non-constant rational function in lowest terms over $\mathbb{F}$. That is, $f=g/h$ with $d := \deg f = \max \left\lbrace \deg g,\deg h \right\rbrace \geq 1$. Being in ``lowest terms'' means $\gcd(g,h)=1$, or equivalently, $g$ and $h$ share no roots in any extension field of $\mathbb{F}$. As such, when referring to zeros and poles of a rational function, we mean roots of its numerator and denominator respectively in an algebraic closure $\overline{\mathbb{F}}$ of $\mathbb{F}$. We recursively define the iterates of $f$ by
\[ f^{(0)}(X)=X, \quad \text{and} \quad f^{(k)} = f \circ f^{(k-1)} \text{ for } k \geq 1. \]
In \cite{Gao}, Gao considers the multiplicative independence of polynomials over finite fields, proving that if $f \in \mathbb{F}_q[X]$ is not a monomial or certain binomial, then the iterates $f^{(1)},\ldots,f^{(n)}$ are multiplicatively independent for $n \geq 1$. Gao uses this fact to give a method for constructing elements of ``high order'' in $\mathbb{F}_{q^n}$ when $q$ is fixed. That is, elements with order larger than any polynomial in $n$ when $n$ is large. In particular, if we define $\bar{n}=q^{\left\lceil \log_qn \right\rceil}$, and $g \in \mathbb{F}_q[X]$ is not a monomial or certain binomial, then any root of an irreducible factor of degree $n$ of $X^{\bar{n}}-g(X)$ is an element in $\mathbb{F}_{q^n}$ of order at least
\[ n^{\frac{\log_qn}{4\log_q(2\log_qn)}-\frac{1}{2}}. \]
Sharper analysis of the same method by Popovych in \cite{Popovych} improves the lower bound on the order to
\[ \begin{pmatrix} n+t-1 \\ t  \end{pmatrix} \prod_{i=0}^{t-1}\frac{1}{d^i}, \]
where $d = \left\lceil 2\log_qn \right\rceil$ and $t = \left\lfloor \log_dn \right\rfloor$.
 
In the case of rational functions over a general field, we also have multiplicative independence of iterates, up to a few exceptional cases. We remark (see Lemma~\ref{NOTMON2}) that these exceptions are precisely the rational functions which, under iteration, eventually become a monomial. For example, if $f^{(n)}(X) = X^k$, then $f^{(n)}(X)$ and $f^{(2n)}(X)=X^{k^2}$ are multiplicatively dependent. Note also that the cases of zero and positive characteristic are different. One distinction, of course, is the existence of inseparable maps in fields of positive characteristic. We see in Lemma~\ref{NOTMON}, that this corresponds to a difference in which rational functions have an iterate which is a polynomial, let alone a monomial. Moreover, especially in the polynomial case, positive characteristic allows terms in iterates to vanish which would otherwise prevent them from becoming monomials.

\begin{theorem} \label{MAIN}
Suppose that $f=g/h \in \mathbb{F}(X)$ has degree $d \geq 2$, and is not a monomial of the form $aX^{\pm d}$, nor of the form $L(X^{p^\ell})$, where $L \in \mathbb{F}(X)$ has degree 1. Let $n \geq 1$, and write
\begin{equation} \label{PSI}
\Psi(n) = \min_{ \substack{k_1,\ldots, k_n \in \mathbb{Z} \\ k_n \neq 0}} \left( \deg \left(\left(f^{(1)} \right)^{k_1} \ldots \left(f^{(n)} \right)^{k_n}\right) \right).
\end{equation}
Then there exists an integer $j \geq 0$ depending only on $f$ such that $\Psi(n) \geq d^n$ if $n \leq j$, and $\Psi(n) \geq d^{n-j}$ if $n > j$.
\end{theorem}

It is easy to show that the above result implies the multiplicative independence of iterates of $f$.

\begin{corollary} \label{IND}
Suppose that $f=g/h \in \mathbb{F}(X)$ has degree $d \geq 2$, and is not of the form $aX^{\pm d}$, or $L(X^{p^\ell})$, where $L \in \mathbb{F}(X)$ has degree 1. Then for any $n \geq 1$, the iterates $f^{(1)}, \ldots, f^{(n)}$ are multiplicatively independent, even up to constants.
\end{corollary}

\begin{proof}
If $(f^{(1)})^{k_1}...(f^{(n)})^{k_n} = c$, $c \in \mathbb{F}$, then Theorem~\ref{MAIN} ensures $k_n=0$, as otherwise the degree would be positive. Then we get $k_{n-1}= \ldots =k_1=0$ recursively. 
\end{proof}

In the polynomial case, we also obtain a lower bound on the number of distinct zeros of a multiplicative combination of iterates.

\begin{theorem} \label{ZEROS}
Suppose $f \in \mathbb{F}[X]$ has degree $d \geq 2$, and has non-vanishing derivative. Let $\z(f)$ denote the number of distinct zeros of $f$ (in an algebraic closure of $\mathbb{F}$), and define
\begin{equation}
Z(n) := \min_{ \substack{k_1,\ldots, k_n \in \mathbb{Z} \\ k_n \neq 0}} \left( \z \left(\left(f^{(1)} \right)^{k_1} \ldots \left(f^{(n)} \right)^{k_n}\right) \right).
\end{equation}
Let $e$ be the least positive integer $k$ such that $f^{(k)}(0)=0$, and say that $e=\infty$ if $f^{(k)}(0) \neq 0$ for all $k \geq 1$. Suppose that $f(0) \neq 0$ and $\z(f)>1$, or that $\z(f) > 2$. Then $Z(n) \geq \gamma(f) d^{n-1} + 1$ if $n \leq e$, and $Z(n) \geq d^{n-e} + 1$ when $n > e$, where
\[ \gamma(f) = \begin{cases} \z(f) - 1, \quad \text{if $\mathbb{F}$ has characteristic 0}, \\ 1, \qquad \qquad \text{otherwise.} \end{cases} \]
\end{theorem}

We use Corollary \ref{IND} in the following extension of the main theorem in \cite{Gao}.

\begin{theorem} \label{ORD}
Let $n \geq 1$, let $g,h \in \mathbb{F}_q[X]$ be coprime with $\deg h, \deg g \leq d = \left\lceil 2\log_qn \right\rceil$, and suppose $f=g/h$ satisfies the conditions from Corollary~\ref{IND}. Suppose that $\alpha \in \mathbb{F}_{q^n}$ has degree $n$ and is a root of $X^mh(X)-g(X)$, where $m = \bar n = q^{\left\lceil \log_qn \right\rceil}$. Then for
$$
s = \begin{cases} n-1, \: \: \quad \qquad f \in \mathbb{F}[X], \\ \lfloor (n-1)/2 \rfloor, \: \: \text{ otherwise,} \end{cases}
$$
and $t = \left\lfloor \log_dn \right\rfloor$, $\alpha$ has order in $\mathbb{F}_{q^n}$ at least \[ \begin{pmatrix} s+t \\ t  \end{pmatrix} \prod_{i=0}^{t-1}\frac{1}{d^i}. \] 
\end{theorem}

As an aside we additionally ask, given rational functions $F_1,\ldots,F_n \in \mathbb{F}(X,Y)$ and polynomial $u \in \mathbb{F}[X]$, when $F_1(X,u(X)),\ldots,F_n(X,u(X))$ are multiplicatively dependent. In particular, we find upper bounds on the degree of $u$ such that this is possible, and the number of monic $u$ for which this is the case.

\begin{theorem} \label{SHIFT}
Suppose $\mathbb{F}$ is a field of characteristic zero, and $F_i = G_i/H_i \in \mathbb{F}(X,Y)$ are rational functions for $1 \leq i \leq n$, of respective degrees $d_{1} \leq \ldots \leq d_{n}$ in $X$ and $1 \leq e_{1} \leq \ldots \leq e_{n}$ in $Y$. For $1 \leq i \neq j \leq n$, define
\[ R_{ij}(X) = \Res_Y(G_i,G_j)\Res_Y(G_i,H_j)\Res_Y(H_i,G_j)\Res_Y(H_i,H_j), \]
where $\Res_Y(P,Q)$ is the resultant of $P,Q \in \mathbb{F}[X,Y]$, considered as polynomials in $Y$, and set 
\[ E = \sum_{1 \leq i < n} \sum_{i < j \leq n} \deg R_{ij}. \] 
If $R_{ij} \not\equiv 0$ for all $i \neq j$, then there are finitely many monic polynomials $u \in \mathbb{F}[X]$ such that 
\[ F_1(X,u(X)), \ldots, F_n(X,u(X)) \]
are multiplicatively dependent. In particular, such a $u$ has degree not exceeding $E+2d_n-1$.
\end{theorem}

Recalling that the resultant of two polynomials of respective degrees $m$ and $n$ is a polynomial in the coefficients of degree $m+n$, and that each $G_i$, written as a polynomial in $Y$, has degree at most $e_n$, with each coefficient having degree not exceeding $d_n$. We have for $i \neq j$, that $\deg \Res_Y(G_i,G_j) \leq (e_n+e_n)d_n=2d_ne_n$. Thus, counting $\frac{n(n-1)}{2}$  distinct pairs $\{i,j\}$, we obtain $E \leq 4n(n-1)d_ne_n$.

Theorem~\ref{SHIFT} can be applied to the particular scenario of shifting a given set of polynomials by a polynomial $u$, giving a analogue of results for algebraic numbers from \cite{BMZ} and \cite{DubSha}.

\begin{corollary} \label{POLYSHIFT}
Suppose $\mathbb{F}$ has characteristic zero and $f_1, \ldots ,f_n \in \mathbb{F}[X]$ are distinct polynomials, not all constant, of respective degrees $d_1 \leq \ldots \leq d_n$ and let 
\[ C = d_n \frac{n(n-1)}{2}. \]
Then there are at most $\binom{2C+3d_n-1}{C}$ monic polynomials $u \in \mathbb{F}[X]$ such that 
\[ f_1+u, \ldots ,f_n+u \]
are multiplicatively dependent. In particular, such a $u$ has degree not exceeding $C+2d_n-1$.
\end{corollary}

The paper is organised with sections corresponding to proofs of the main theorems: In the next section, we collect various results on iterates of rational functions, specifically concerning zeros and poles which are common to different iterates, and the degrees of the numerator and denominator of iterates. We use these results to bound from below the number (counted with multiplicity) of zeros and poles of a given iterate which cannot be found in any of the previous ones. We thus obtain Theorem~\ref{MAIN}. In Section~\ref{sec:zeros}, we give the proof of a version of \cite[Main~Theorem]{FucPet}, which holds for polynomials over fields of arbitrary characteristic. This is used in conjunction with the general method from Section~\ref{sec:main} to prove Theorem~\ref{ZEROS}. In Section~\ref{sec:ord}, we discuss elements of high order in finite fields in a manner analogous to \cite{Gao,Popovych}, but in a slightly more general setting. Finally, in Section~\ref{sec:shift}, we use resultants in conjunction with the polynomial $ABC$-theorem to prove Theorem~\ref{SHIFT}.

\section{Proof of Theorem~\ref{MAIN}} \label{sec:main}

To prove Theorem~\ref{MAIN}, we need some facts about the composition of rational functions. Let $u=v/w, F=G/H \in \mathbb{F}(X)$ be in lowest terms over $\mathbb{F}$, chosen so $H$ is monic and $G$ has leading coefficient $A$, and write 
\[ u(X) = \frac{v(X)}{w(X)} = \frac{a_lX^l+ \ldots +a_sX^s}{b_mX^{m}+ \ldots +b_tX^t}, \: a_l,a_s,b_m,b_t \neq 0, \] 
with $\deg u \geq 1$. Let $u \circ F = P/Q$. We have
\begin{align}
\frac{P(X)}{Q(X)} & = \frac{a_l\left(\frac{G(X)}{H(X)}\right)^l+ \ldots +a_s\left(\frac{G(X)}{H(X)}\right)^s}{b_m\left(\frac{G(X)}{H(X)}\right)^m+ \ldots +b_t\left(\frac{G(X)}{H(X)}\right)^t} \notag \\
& = H(X)^{m-l}G(X)^{s-t}\frac{q(X)}{r(X)}, \label{eqn1}
\end{align}
where 
\[ q(X) = \sum_{i=0}^{l-s} a_{l-i}G(X)^{l-s-i}H(X)^{i} \: \text{ and } \: r(X) = \sum_{i=0}^{m-t} b_{m-i}G(X)^{m-t-i}H(X)^{i}. \] 
Note that a composition of rational functions in lowest terms is itself in lowest terms (\cite[Lemma~2.2]{Carter} is easily extended to our situation). In particular, $G$, $H$, $q$ and $r$ are pairwise relatively prime. This means we need not worry about the possibility of factors cancelling after composition. Hence, from \eqref{eqn1}, whenever $\deg G \neq \deg H$ we have
\begin{align}
\deg P & = \deg H(\deg u - l)+(\deg G)s+\deg F(l-s), \label{eqn2} \\
\deg Q & = \deg H(\deg u - m)+(\deg G)t+\deg F(m-t). \label{eqn3}
\end{align}
Moreover, when $\deg G=\deg H$, the coefficient of $X^{l \deg u}$ is $v(A)$ in $P$, and $w(A)$ in $Q$. These can't both be zero as $gcd(v,w)=1$, so in all cases we have
\begin{equation}
\deg u \circ F = (\deg u)(\deg F). \label{DEGG}
\end{equation}

We can use these facts to obtain results about which zeros and poles are common to different iterates of $f$, beginning by extending a result of Gao~\cite[Lemma~2.2]{Gao}.

\begin{lemma} \label{COMROOTS}
Write $f^{(k)}=g_k/h_k$ for the $k$-th iterate of $f$, and let $e$ be defined as in Theorem \ref{ZEROS}. Further define $\epsilon$, $\mu$ and $\nu$ to be respectively the smallest positive integers $k$ such that $h_k(0)=0$, $\deg g_k < \deg h_k$, and $\deg g_k > \deg h_k$ (these again take the value $\infty$ if their respective conditions are not satisfied for any $k \geq 1$). Then, for all $k > \ell \geq 1$,
\begin{enumerate}[label=(\roman*)]
\item A zero of $f^{(\ell)}$ is a zero of $f^{(k)}$ if and only if $e < \infty$ and $k \equiv \ell \pmod{e}$.
\item A pole of $f^{(\ell)}$ is a pole of $f^{(k)}$ if and only if $\deg g_{k-\ell} > \deg h_{k-\ell}$.
\item A pole of $f^{(\ell)}$ is a zero of $f^{(k)}$ if and only if $\deg g_{k-\ell} < \deg h_{k-\ell}$.
\item If $\mu < \nu$, then a zero of $f^{(\ell)}$ is a pole of $f^{(k)}$ if and only if $\epsilon < \infty$ and $k \equiv \ell - \mu \pmod{e}$. Note that here, $e=\epsilon+\mu$.
\end{enumerate} 
\end{lemma}

\begin{proof}
Let $k > \ell \geq 1$. For part (i), suppose that a zero $\alpha$ of $f^{(\ell)}$ is a zero of $f^{(k)}$. Then $f^{(k)}(\alpha)=f^{(\ell)}(\alpha)=0$. As $f^{(k)}=f^{(k-\ell)} \circ f^{(\ell)}$, we have
\[ f^{(k-\ell)}(0)=f^{(k-\ell)}\left(f^{(\ell)}(\alpha)\right)=f^{(k)}(\alpha)=0. \]
Thus we must have $e < \infty$, so assume this is the case. If $k \equiv \ell \pmod{e}$, say $k = \ell+je$ where $j \geq 1$, then for any zero $\beta$ of $f^{(\ell)}$,
\[ f^{(k)}(\beta)=f^{(je)}\left(f^{(\ell)}(\beta) \right)=f^{(je)}(0)=0. \]
Hence any zero of $f^{(\ell)}$ is a zero of $f^{(k)}$. Now, suppose $k \not\equiv \ell \pmod{e}$, say $k=\ell+je+r$ where $u \geq 0$ and $1 \leq r < e$. If $f^{(k)}$ and $f^{(\ell)}$ have a zero in common then, by the above argument, $f^{(je+r)}(0)=f^{(k-\ell)}(0)=0$. But then
\[ f^{(r)}(0) = f^{(r)}(f^{(je)}(0)) = f^{(je+r)}(0)=0, \]
contradicting the choice of $e$. Therefore $f^{(k)}$ and $f^{(\ell)}$ have no zero in common when $k \not\equiv \ell \pmod{e}$.

Writing $f^{(k)}=f^{(k-\ell)} \circ f^{(\ell)}$, the second and third parts follow immediately from \eqref{eqn1}.

Now, suppose that $\mu < \nu$. By definition, we have that $\deg g_k = \deg h_k$ for $1 \leq k < \mu$. Set $u=f^{(j)}$, $F=f^{(\mu)}$, so $f^{(\mu+j)}=u \circ F = P/Q$ as in \eqref{eqn1}. If $e, \epsilon > j \geq 1$, then $s=t=0$, and so \eqref{eqn2} and \eqref{eqn3} give $\deg g_{\mu+j} = \deg h_{\mu+j} = d^{\mu+j}$. We thus note that
\begin{equation} \label{DEG}
\deg g_k = \deg h_k = d^k \quad \text{for all $1 \leq k \neq \mu < \mu + \min \{\epsilon, e \}$}.
\end{equation}
Suppose a zero $\alpha$ of $f^{(\ell)}$ is a pole of $f^{(k)}$. Then we have 
\[ f^{(k-\ell)}(0)=f^{(k-\ell)} \left( f^{(\ell)} (\alpha) \right) = f^{(k)}(\alpha), \]
and so $0$ is a pole of $f^{(k-\ell)}$. That is, we indeed have $\epsilon < \infty$. Furthermore, if $e < \epsilon$, then $f^{(\epsilon-e)}(0)=f^{(\epsilon-e)} \left( f^{(e)}(0) \right) = f^{(\epsilon)}(0)$, so $0$ is a pole of $f^{(\epsilon-e)}$, contradicting the choice of $\epsilon$. Hence we have $\epsilon < e$, and by setting $u = f^{(j)}$, $F = f^{(\epsilon)}$, \eqref{eqn1} gives that $0$ is a zero of $f^{(\epsilon+j)}$ if and only if $\deg g_j < \deg h_j$. Thus $e = \epsilon+\mu$. If $k \equiv \ell - \mu \pmod{e}$, say $k = \ell+je-\mu = \ell + (j-1)e + \epsilon$, with $j \geq 1$, then for any zero $\beta$ of $f^{(\ell)}$,
\[
f^{(k)}(0) =f^{(\epsilon)} \left( f^{((j-1)e)} \left(f^{(\ell)}(\beta) \right) \right)
= f^{(\epsilon)} \left( f^{((j-1)e)}(0) \right) = f^{(\epsilon)}(0).
\]
Thus, any zero of $f^{(\ell)}$ is a pole of $f^{(k)}$. Suppose now that $k = \ell + je + r-\mu$, with $j \geq 1$ and $1 \leq r < e$. If a zero $\beta$ of $f^{(\ell)}$ is a pole of $f^{(k)}$, then $f^{(k-\ell)}(0)=f^{(k)}(\beta)$, and so $0$ is a pole of $f^{(k-\ell)}=f^{((j-1)e+\epsilon+r)}$. Since
\[ 
f^{((j-1)e+\epsilon)}(0) = f^{(\epsilon)} \left( f^{((j-1)e)}(0) \right) = f^{(\epsilon)}(0),
\]
$0$ is also a pole of $f^{((j-1)e+\epsilon)}$ and hence, by part (ii), $\deg g_r > \deg h_r$. This is a contradiction, since from \eqref{DEG} and the definition of $\mu$, $\deg g_k \leq \deg h_k$ for all $1 \leq k < \mu + \min \{ \epsilon,e \} = \mu + \epsilon = e$.
\end{proof}

We may also determine facts about the degrees of iterates of $f$. 

\begin{lemma} \label{DEGREES}
Throughout, if $\min\{\mu, \nu \} < \infty$, define 
\[ \delta = | \deg g_{ \min \{ \mu, \nu \} } - \deg h_{ \min \{ \mu, \nu \} } |, \]
and let $S_k$ and $T_k$ be respectively the degrees of the lowest order term in $g_k$ and $h_k$. We have
\begin{enumerate}[label=(\roman*)]
\item If $\nu < \mu$, then for any integer $i \geq 1$, $\deg g_{i\nu} = d^{i\nu}$, and $\deg h_{i\nu} = d^{i\nu}-\delta^i$.
Moreover, $\deg g_k = \deg h_k = d^k$ whenever $k \not \equiv 0 \pmod{\nu}$.
\item If $\mu < \nu$ and $\epsilon=e=\infty$, then $\deg g_k = \deg h_k = d^k$ for all $k \neq \mu$.
\item Let $\mu < \nu$, $e < \epsilon$, and write $S_e=S$. Then, if $k = ie+\mu$ for some integer $i \geq 0$, $\deg g_k = d^k-\delta S^i$ and $\deg h_k = d^k$. Otherwise, we have $\deg g_k = \deg h_k = d^k$.
\item Let $\mu < \nu$ and $\epsilon < \infty$. Recall then, from Lemma~\ref{COMROOTS}~(iv), that $e = \epsilon+\mu$, and write $T_{\epsilon} = T$. Then $\deg g_{\mu + k} = d^{\mu+k}-\delta S_k$ and $\deg h_{\mu+k} = d^{\mu+k}-\delta T_k$ for any $k \geq 1$. In particular, if $k = ie$, then $S_k = \delta^i T^i$ and $T_k=0$; if $k = ie+\epsilon$, then $S_k = 0$ and $T_k = \delta^i T^{i+1}$; otherwise, $S_k=T_k=0$.
\end{enumerate}
\end{lemma}

\begin{proof}
Throughout the proof, we will write a given iterate $f^{(k)}=u \circ F=P/Q$, and infer the degrees of its numerator and denominator via the equations \eqref{eqn2} and \eqref{eqn3}. By definition and from \eqref{DEGG}, $\deg g_k = \deg h_k = d^k$ for $1 \leq k < \nu$, and we have $\deg g_\nu = d^{\nu}$ and $\deg h_\nu = d^{\nu}-\delta$. Let $i \geq 1$ and suppose that $\deg g_{i\nu} = d^{i\nu}$ and $\deg h_{i\nu} = d^{i\nu}-\delta^i$. Setting $u=f^{(k)}$ and $F=f^{(i\nu)}$, we obtain $\deg g_{i\nu+k} = \deg h_{i\nu+k}=d^{i\nu+k}$ when $1 \leq k < \nu$, and when $k=\nu$, we get $\deg g_{(i+1)\nu} = d^{(i+1)\nu}$ and
\begin{align*}
\deg h_{(i+1)\nu} & = (d^{\nu}-\delta)(d^{(i-1)\nu}-\delta^{i-1})\delta+d^{i\nu} (d^{\nu}-\delta) \\
& = d^{i\nu}-\delta^i.
\end{align*}
We thus obtain part (i) by induction. The second part follows from \eqref{DEG}. 

For the third and fourth parts, setting $u = f^{(k)}$ and $F=f^{(\mu)}$ gives 
\[ \deg g_{k+\mu} = d^{\mu}(d^k-l) + (d^{\mu}-\delta)S_k+d^{\mu}(l-S_k) = d^{k+\mu}-\delta S_k \]
and likewise $\deg h_{k+\mu} = d^{k+\mu}-\delta T_k$. If we put $u = f^{(e)}$, $F=f^{((i-1)e)}$, induction on $i$ with \eqref{eqn1} shows that $S_{ie}=S_e^i$. Also, by Lemma~\ref{COMROOTS}~(i), $S_k = 0$ for all $k \not\equiv 0 \pmod{e}$. When $e < \epsilon = \infty$, $T_k =0$ for all $k$, which proves (iii). 

For part (iv), we set $u=f^{(\mu)}$ and $F=f^{(\epsilon)}$ so that \eqref{eqn1} gives $S_e = \delta T$, and thus $S_{ie} = \delta^i T^i$. We similarly obtain $T_{ie+\epsilon} = \delta^i T^{i+1}$. Finally, if $k \not \equiv \epsilon = e - \mu \pmod{e}$, then $T_k = 0$ by Lemma~\ref{COMROOTS}~(iv), as required.
\end{proof}

We hence obtain the following result.

\begin{lemma} \label{DEGREES2}
Suppose $\mu < \nu$ and $\epsilon < \infty$, and let $1 \leq \ell < k$.
\begin{enumerate}[label=(\roman*)]
\item A zero or pole of $f^{(\ell)}$ is a zero of $f^{(k)}$ if and only if it is a pole of $f^{(k-\mu)}$.
\item A zero or pole of $f^{(\ell)}$ is a pole of $f^{(k)}$ if and only if it is a zero of $f^{(k-\epsilon)}$.
\end{enumerate}
\end{lemma}

\begin{proof}
For the first part, by Lemma~\ref{COMROOTS}~(i) we have that a zero of $f^{(\ell)}$ is a zero of $f^{(k)}$ if and only if $k \equiv \ell \pmod{e}$. Then, by Lemma~\ref{COMROOTS}~(iv), a zero of $f^{(\ell)}$ is a pole of $f^{(k-\mu)}$ if and only if $k-\mu \equiv \ell - \mu \pmod{e}$, which is an equivalent condition. From Lemma~\ref{COMROOTS}~(iii), a pole of $f^{(\ell)}$ is a zero of $f^{(k)}$ if and only if $\deg g_{k-\ell} < \deg h_{k-\ell}$. This occurs precisely when $k - \ell - \mu \equiv 0 \pmod{e}$ by Lemma~\ref{DEGREES}~(iv). On the other hand, a pole of $f^{(\ell)}$ is a pole of $f^{(k-\mu)}$ if and only if $\deg g_{k-\ell-\mu} > \deg h_{k-\ell-\mu}$. By Lemma~\ref{DEGREES}~(iv), this happens exactly when $k-\mu \equiv \ell \pmod{e}$, which is again equivalent.

For part (ii), by Lemma~\ref{COMROOTS}~(iv), a zero of $f^{(\ell)}$ is a pole of $f^{(k)}$ if and only if $k \equiv \ell - \mu \pmod{e}$. Since $e = \mu + \epsilon$, this is equivalent to $k-\epsilon \equiv \ell \pmod{e}$, which is the precise condition for a zero of $f^{(\ell)}$ to be a zero of $f^{(k-\epsilon)}$, by Lemma~\ref{COMROOTS}~(i). Furthermore, from Lemma~\ref{COMROOTS}~(ii), a pole of $f^{(\ell)}$ is a pole of $f^{(k)}$ if and only if $\deg g_{k-\ell} > \deg h_{k-\ell}$. According to Lemma~\ref{DEGREES}~(iv), this is equivalent to $k-\ell$  being of the form $\mu + ie + \epsilon$, which equates to $k-\ell-\epsilon = \mu+ie$. Again by Lemma~\ref{DEGREES}~(iv), this is equivalent to having $\deg g_{k-\ell-\epsilon} < \deg h_{k-\ell-\epsilon}$, which is in turn equivalent to the given pole of $f^{(\ell)}$ being a zero of $f^{(k-\epsilon)}$, by Lemma~\ref{COMROOTS}~(iii).
\end{proof}

As we remarked in the introduction, in order to prove multiplicative independence for the iterates of $f$, it is clearly necessary to show that no iterate of $f$ is a monomial. We first look to a result of Silverman \cite{Silverman2}. Recall that two rational functions $\phi,\psi$ are \emph{linearly conjugate} if there exists a rational function $u$ of degree 1 such that $\phi = u^{-1} \circ \psi \circ u$.

\begin{lemma} \label{NOTMON}
Suppose there exists a positive integer $n$ such that $f^{(n)} \in \mathbb{F}[X]$. Then either $f \in \mathbb{F}[X]$, $f$ is separable and linearly conjugate to $1/X^d$, or $f$ is not separable and $f(X)=L( X^{p^{\ell}})$ for some $L \in \mathbb{F}(X)$ of degree 1.
\end{lemma}

Indeed, if no iterate of $f$ is a polynomial, then certainly none can be a monomial. In fact, in the case where $f$ is separable, we show that a rational function has a monomial iterate if and only if it is itself a monomial. This is not true however, when $f$ is not separable. For example, if $\mathbb{F}$ has characteristic 2, then $f(X) = 1+1/X^2$ satisfies $f^{(2)}(X) = \frac{1}{X^4+1}$ and $f^{(3)}(X) = X^8$.

Note that in the case of characteristic 0, some cases of the following can actually be viewed as a corollary of the stronger result \cite[Theorem~1]{Zannier}, which concerns the number of terms (monomials) of composite polynomials. The results of \cite{Zannier} are further extended to rational functions in \cite{Fuchs}.

\begin{lemma} \label{NOTMON2}
If $f \in \mathbb{F}(X)$ is neither a monomial, nor of the form $L(X^{p^\ell})$ for some $\ell \geq 0$ and $L \in \mathbb{F}(X)$ of degree 1, then $f^{(n)}$ is not a monomial for any $n \geq 1$.
\end{lemma}

\begin{proof}
We begin with the case where $f \in \mathbb{F}[X]$ is a polynomial. First suppose $\mathbb{F}$ has zero characteristic. We proceed by induction on $k$. That is, suppose $\deg f \geq 2$, and that $f$ is not a monomial. Then the case where $k=1$ is trivial. If $f^{(k-1)}$ is not a monomial, we can write
\begin{align*}
f(X) & = a_1X^{d_1}+ \ldots +a_sX^{d_s}; \\
s & >1, \: d=d_1> \ldots >d_s \geq 0, \: a_1, \ldots ,a_s \in \mathbb{F} \setminus \left\lbrace 0 \right\rbrace,
\end{align*}
and
\begin{align*}
f^{(k-1)}(X) & = b_1X^{e_1}+ \ldots +b_tX^{e_t}; \\
t & > 1, \: d^{k-1}=e_1> \ldots >e_t \geq 0, \: b_1, \ldots ,b_t \in \mathbb{F} \setminus \left\lbrace 0 \right\rbrace.
\end{align*}
Hence we have the following cases:

If $d_s=0$, $e_t \neq 0$, we have that
\begin{align*}
f^{(k)}(X) & = f(f^{(k-1)}(X)) \\
& = a_1(b_1X^{e_1}+ \ldots +b_tX^{e_t})^{d_1}+ \ldots +a_s
\end{align*}
has constant term $a_s \neq 0$. Similarly, if $d_s \neq 0$, $e_t=0$,
\begin{align*}
f^{(k)}(X) & = f^{(k-1)}(f(X)) \\
& = b_1(a_1X^{d_1}+ \ldots +a_sX^{d_s})^{e_1}+ \ldots +b_t
\end{align*}
has constant term $b_t \neq 0$. If $d_s \neq 0$, $e_t \neq 0$, then
\begin{align*}
f^{(k)}(X) & = f(f^{(k-1)}(X)) \\
& = a_1(b_1X^{e_1}+ \ldots +b_tX^{e_t})^{d_1}+ \ldots +a_s(b_1X^{e_1}+ \ldots +b_tX^{e_t})^{d_s}
\end{align*}
has lowest order term $a_sb_t^{d_s}X^{d_se_t} \neq 0$, since $a_s \neq 0$, $b_t \neq 0$. Finally, when $d_s=e_t=0$, if $e_2 > 0$, we have
\begin{align*}
f^{(k)}(X) & = f(f^{(k-1)}(X)) \\
& = a_1(b_1X^{e_1}+b_2X^{e_2}+ \ldots +b_t)^{d_1}+ \ldots +a_s.
\end{align*}
In this case, the term in $X^{(d_1-1)e_1+e_2}$ has coefficient $d_1a_1b_1^{d_1-1}b_2 \neq 0$, since we have $a_1,b_1,b_2 \neq 0$, and $\mathbb{F}$ has $0$ characteristic. Otherwise, $e_2 = 0$ and
\begin{align*}
f^{(k)}(X) & = f^{(k-1)}(f(X)) \\
& = b_1(a_1X^{d_1}+a_2X^{d_2}+ \ldots +a_s)^{e_1}+b_2.
\end{align*}
Similarly, the term in $X^{(e_1-1)d_1+d_2}$ has coefficient $e_1b_1a_1^{e_1-1}a_2 \neq 0$. That is, in all cases $f^{(k)}$ is not a monomial, and we are done.

Now, suppose $\mathbb{F}$ has positive characteristic $p$, and that $f^{(k)}$ is monomial, say of the form $cX^{d^k}$ with $c \in \mathbb{F} \setminus \left\lbrace 0 \right\rbrace$, for some $k > 1$. We can write 
\[ f(X) = a_1X^{d_1p^{\ell}}+ \ldots +a_tX^{d_tp^{\ell}}+b, \] 
where $a_1, \ldots ,a_t \in \mathbb{F} \setminus \left\lbrace 0 \right\rbrace$, $b \in \mathbb{F}$, $t \geq 1$, $\ell \geq 0$, $d_1 > \ldots > d_t \geq 1$, and $p \nmid \gcd(d_1, \ldots ,d_t)$.

Here, the degree of $f$ is $d = d_1p^{\ell}$. Denote $r = p^{\ell}$ and let
\begin{align*}
v(X) & = a_1X^{d_1}+ \ldots +a_tX^{d_t}+b, \\
w_i(X) & = a_1^{r^{-i}}X^{d_1}+ \ldots +a_t^{r^{-i}}X^{d_t}+b^{r^{-i}}, \: i \geq 1.
\end{align*}
Since $r^i$ is a power of $p$, we have for any $i \geq 1$ \[ (w_i(X))^{r^i} = a_1X^{d_1r^i}+ \ldots +a_tX^{d_tr^i}+b = v(X^{r^i}). \]
Hence
\begin{align*}
f(X) & = v(X^r),\\
f^{(2)}(X) & = v(v(X^r)^r))=v\left((w_1(X))^{r^2}\right) = (w_2 \circ w_1(X))^{r^2}. \\
& \vdots \\
f^{(k)}(X) & = (w_k \circ w_{k-1} \circ \ldots \circ w_1(X))^{r^k},\: \: k \geq 1.
\end{align*}
Hence we have
\[ w_k \circ w_{k-1} \circ \ldots \circ w_1(X) = c_0X^{d_1^k}, \] 
where $c_0 = c^{r^{-k}} \neq 0$, since $c \neq 0$. Differentiating then gives 
\begin{multline} \label{eqn6}
w_k'(w_{k-1} \circ \ldots \circ w_1(X)) \cdot w_{k-1}'(w_{k-2} \circ \ldots \circ w_1(X)) \cdots w_2'(w_1(X)) \cdot w_1'(X) \\ = d_1^kc_0X^{d_1^k-1}.
\end{multline}
Since $p \nmid \gcd(d_1, \ldots ,d_t)$, $w_i' \neq 0$ for all $i \geq 1$. Thus, the polynomial on the left hand side of \eqref{eqn6} is not zero. So $p \nmid d_1$, as otherwise the right hand side would be zero. Since $d_1^kc_0 \neq 0$, the equation \eqref{eqn6} implies that $w_1'(X)$ divides $X^{d_1^k-1}$. Therefore $w_1'$ is a monomial. Since $p \nmid d_1$, we must have $p \mid d_i$ for $2 \leq i \leq t$. Hence 
\[ w_i'(X) = d_1a_1^{-r^i}X^{d_1-1}, \: \: i \geq 1. \] 
From $(6)$, $w_2'(w_1(X)) = d_1a_1^{-r^2}(w_1(X))^{d_1-1}$ is also a factor of $x^{d_1^k-1}$. If $d_1 > 1$, then $w_1$ is a monomial and hence $f$ must also be a monomial. If $d_1= 1$, then $d_1> \ldots >d_t \geq 1$ implies that $t=1$. Therefore $f$ is a binomial of the form $aX^{p^{\ell}}+b$.

Now, suppose $f \notin \mathbb{F}[X]$, and that $f^{(n)}$ is a monomial for some $n \geq 1$. Then in particular, some iterate of $f$ is a polynomial.

If $f$ is separable, then by Lemma \ref{NOTMON}, $f$ is linearly conjugate to $1/X^d$. That is, $f$ has the form
$$
f(X) = a+ \frac{b}{(X-a)^d}, \qquad a,b \in \mathbb{F}.
$$
Then $f^{(2)}(X) = a+b^{1-d}(X-a)^{d^2}$, which is a monomial if and only if $a = 0$, in which case $f$ is a monomial. Suppose $a \neq 0$. Since $f$ is separable, $d \neq p^{\ell}$ for any $\ell > 0$, and so by the above argument, $f^{(n)}$ is not a monomial for any even $n \geq 2$ unless $f$ is a monomial. Moreover, we have in this case $\nu = 2 < \mu$, so by Lemma \ref{DEGREES} (i), $\deg g_n = \deg h_n$, and so $f^{(n)}$ is not a monomial, for all odd $n$.

Finally, if $f$ is not separable, then by Lemma \ref{NOTMON}, $f^{(n)}$ is not a polynomial, and hence is not a monomial, for any $n \geq 1$ unless $f$ is of the form $L(X^{p^\ell})$ for some $L \in \mathbb{F}(X)$ of degree 1.
\end{proof}

We can now prove Theorem~\ref{MAIN}. Recall that we write $f^{(k)}=g_k/h_k$, and define $\delta, S_k$, and $T_k$ as in Lemma~\ref{DEGREES}, again setting $S=S_e$ and $T=T_\epsilon$ where applicable. Now, where $\Psi(n)$ is defined as in \eqref{PSI}, noting that $\mathbb{F}(X)$ is a unique factorisation domain, any zeros or poles of $f^{(n)}$ which can not be found in previous iterates will contribute to the value of $\Psi(n)$ counting multiplicity, since $k_n \neq 0$. 

We first consider the case where $\nu \leq \mu$. Then $\deg g_k \geq \deg h_k$ for all $k$ by Lemma 2.2 (i). Hence $\gcd(g_n,h_k)=1$ for any $k < n$ by Lemma \ref{COMROOTS} (iii). Moreover, if $n \leq e$, then $\gcd(g_n,g_k)=1$ for any $k < n$, by Lemma \ref{COMROOTS} (i). In this case, we have $\Psi(n) \geq \deg g_n = d^n$. Suppose $e < \infty$ and $n > e$. Then for $k < n$, a zero of $f^{(k)}$ is a zero of $f^{(n)}$ if and only if $k \equiv n \pmod{e}$ by Lemma \ref{COMROOTS}. In this case we also have $k \equiv n-e \pmod{e}$, and so such a zero must also be a zero of $f^{(n-e)}$. Write $u=f^{(e)}$ and $F=f^{(n-e)}$, so $\eqref{eqn1}$ gives $g_n = g_{(n-e)}^{S} q$, where $S > 0$ and $\gcd(q,g_{(n-e)}) = 1$. Since $f^{(e)}$ is not a monomial by Lemma \ref{NOTMON2}, we have $S < d^e$, and so $\Psi(n) \geq \deg q = d^n - Sd^{n-e} \geq d^{n-e}$.

Now, suppose $\mu < \nu$. If $n \leq \mu$, then $\gcd(h_n,g_k)=\gcd(h_n,h_k)=1$ for all $k < n$ by Lemma \ref{COMROOTS} (ii) and (iv). Hence $\Psi(n) \geq \deg h_n = d^n$. So suppose $n > \mu$. If $e < \epsilon$, then by Lemma~\ref{DEGREES} (ii) and (iii), $\deg h_k = d^k \geq \deg g_k$ for all $k \geq 1$. Moreover, if $n \leq \epsilon$, then $\deg h_k = d^k \geq \deg g_k$ for all  $1 \leq k \leq n$ by \eqref{DEG}. So, by Lemma~\ref{COMROOTS} (ii) and (iv), $\gcd(g_k,h_n)=\gcd(h_k,h_n)=1$ for all $1 \leq k < n$, giving $\Psi(n) \geq \deg h_n = d^n$. We hence assume that $\epsilon < n < \infty$.

We now split into a further two cases. Firstly, suppose that $\deg g_{\mu} > 0$, so that $\delta < d^{\mu}$. Since $e = \mu + \epsilon > \mu$, we do not have $\mu = ie$, and so $S_{\mu}=0$, by Lemma~\ref{DEGREES}~(iv). Hence, where $u=f^{(\mu)}$ and $F=f^{(n-\mu)}$, \eqref{eqn1} gives $g_n = h_{n-\mu}^{\delta} q$. If $n = \mu + ie$, then $n - \mu = \mu + (i-1)e + \epsilon$, and so by Lemma~\ref{DEGREES}~(iv),
\begin{align*}
\delta \deg h_{n-\mu} + (\deg g_{\mu}) d^{n-\mu} & = \delta ( d^{n-\mu}-\delta^iT^i ) + (d^{\mu}-\delta)d^{n-\mu} \\
& = d^n - \delta^{i+1}T^i = \deg g_n.
\end{align*}
Otherwise, again by Lemma~\ref{DEGREES}~(iv), $\deg g_n = d^n$, and so
\[ \delta \deg h_{n-\mu} + (\deg g_{\mu}) d^{n-\mu} \leq \delta d^{n-\mu} + (d^{\mu}-\delta)d^{n-\mu} = d^n = \deg g_n. \]
Hence, $\deg q \geq (\deg g_{\mu})d^{n-\mu} \geq d^{n-\mu}$. Moreover, we have $\gcd(h_k,q)=\gcd(g_k,q)=1$ for all $1 \leq k < n$ by Lemma~\ref{DEGREES2}~(ii), and therefore $\Psi(n) \geq \deg q \geq d^{n-\mu}$.

On the other hand, where $\deg g_{\mu}=0$, we set $u = f^{(\epsilon)}$, and $F=f^{(n-\epsilon)}$. If $\epsilon \leq \mu$, then by definition $\deg g_{\epsilon} \leq \deg h_{\epsilon}$. Otherwise, $\epsilon=\mu+k$, with $k \neq ie, ie+\epsilon$, and so by Lemma~\ref{DEGREES}~(iv), we have $\deg g_{\epsilon} = \deg h_{\epsilon}$. Hence, by \eqref{eqn1}, $f^{(n)}=h_{n-\epsilon}^{m-l} g_{n-\epsilon}^{-T} q/r$, where $m \geq l$. We thus obtain $\deg r = \deg h_n - T \deg g_{n-\epsilon}$. Note that $T < d^{\epsilon}$, as if this were not the case, by Lemma~\ref{DEGREES}~(iv) we would have 
\[ \deg h_{\mu+\epsilon} = d^{\mu+\epsilon} - \delta T = d^{\mu+\epsilon}-d^{\mu} d^{\epsilon} = 0, \]
and $S_{\mu+\epsilon}=S_e=\delta T = d^{\mu} d^{\epsilon}$, which implies that $f^{(\mu+\epsilon)}$ is a monomial, contradicting Lemma \ref{NOTMON2}. In particular, this means that $d^n - Td^{n-\epsilon} \geq d^{n-\epsilon}$. Hence, if $n = \mu + ie + \epsilon$, then $n-\epsilon = \mu + ie$, so by Lemma~\ref{DEGREES}~(iv), we have
\[ \deg r = d^n - \delta^{i+1}T^{i+1} - T(d^{n-\epsilon} - \delta^{i+1}T^i) = d^n-Td^{n-\epsilon} \geq d^{n-\epsilon}. \]
Otherwise, once again using Lemma~\ref{DEGREES}~(iv), $\deg h_n = d^n$, and so
\[ \deg r = d^n - T \deg g_{n-\epsilon} \geq d^n - T d^{n-\epsilon} \geq d^{n-\epsilon}. \]
To conclude, by Lemma~\ref{DEGREES2}~(iii), we have that $\gcd(h_k,r)=\gcd(g_k,r)=1$ for all $1 \leq k < n$, and thus $\Psi(n) \geq \deg r \geq d^{n-\epsilon}$. This completes the proof. \QEDB

\section{Proof of Theorem~\ref{ZEROS}} \label{sec:zeros}

Recall the polynomial $ABC$-theorem (proved first by Stothers \cite{Stothers}, then independently by Mason \cite{Mason} and Silverman \cite{Silverman}).

\begin{lemma} \label{ABC}
Let $\mathbb{F}$ be a field and let $A,B,C \in \mathbb{F}[X]$ be relatively prime polynomials such that $A+B+C=0$ and not all of $A, B$ and $C$ have vanishing derivative. Then 
\[ \max \left\lbrace \deg A, \deg B, \deg C \right\rbrace \leq \deg \rad(ABC) -1, \] 
where, for $f \in \mathbb{F}[X]$, $\rad(f)$ is the product of the distinct monic irreducible factors of $f$.
\end{lemma}

We use this to obtain a version of part of the main result of \cite{FucPet}. Namely, we give a lower bound for the number of distinct zeros of a composite polynomial.

\begin{lemma} \label{COMPOSZERO}
Let $f = g \circ h \in \mathbb{F}[X]$, where $h$ has non-vanishing derivative, and $\z(g) > 1$. Then
\[ \z(f) \geq \gamma(g) \deg h + 1, \]
where $\gamma$ is defined as in Theorem~\ref{ZEROS}.
\end{lemma}

\begin{proof}
In the characteristic 0 case, this is readily obtained from the proof of \cite[Main~Theorem]{FucPet}. In particular, we are in the case where $v_{\infty}(g) \neq 0$, where $v_{\infty}$ is the non-archimedean valuation defined on $\mathbb{F}(X)$ by $v_{\infty}(p/q)= \deg p - \deg q$. When the characteristic is positive, we proceed in much the same vein. Write
$$
f(X) = \prod_{i=1}^n (X-\alpha_i)^{f_i}, \quad g(X) = \prod_{j=1}^t (X-\beta_j)^{k_j}.
$$
Then
$$
f(X) = g(h(X)) = \prod_{j=1}^t (h(X)-\beta_j)^{k_j}.
$$
For $\beta_i \neq \beta_j$, the factors $h(X)-\beta_i$ and $h(X)-\beta_j$ have no zeros in common, so $t \leq n$, and there exists a partition of $\{1,\ldots,n\}$ into disjoint subsets $S_{\beta_1},\ldots, S_{\beta_t}$, such that
$$
h(X) - \beta_j = p_j(X) := \prod_{m \in S_{\beta_j}} (X-\alpha_m)^{l_m},
$$
with $l_m k_m = f_m$, for every $j=1,\ldots,t$. Since $t = \z(g) > 1$, we can take $1 \leq i < j \leq t$, and obtain $h(X) = \beta_i + p_i(X) = \beta_j + p_j(X)$. That is,
$$
(\beta_i - \beta_j) + p_i + (-p_j) = 0,
$$
where the polynomials on the left-hand side are relatively prime, and in particular, since $h$ has non-vanishing derivative, so does $p_i$. Thus, applying Lemma \ref{ABC}, we have
\begin{align*}
\max \{ \deg (\beta_i-\beta_j), \deg p_i, \deg (-p_j) \} & = \deg h \\
& \leq \deg \rad( (\beta_j-\beta_i)p_ip_j ) - 1 \leq n-1.
\end{align*}
Therefore $n = \z(f) \geq \deg h + 1$. 
\end{proof}

We now prove Theorem \ref{ZEROS}. Suppose $f \in \mathbb{F}[X]$ has non-vanishing derivative. Then for any positive integer $n$,
\[ \frac{d}{dX} f^{(n)}(X) = f'(f^{(n-1)}(X)) \cdot f'(f^{(n-2)}(X)) \cdots f'(f(X)) \cdot f'(X) \neq 0. \]
We can hence apply Lemma~\ref{COMPOSZERO} to obtain $\z(f^{(n)}) \geq \gamma(f) d^{n-1} + 1$. As in the proof of Theorem \ref{MAIN}, any zeros of $f^{(n)}$ which cannot be found in previous iterates will contribute to the value of $Z(n)$, but this time without multiplicity. If $n \leq e$, then $\gcd( f^{(k)},f^{(n)})=1$ for all $1 \leq k <n$ by Lemma \ref{COMROOTS} (i), and so $Z(n) \geq \z(f^{(n)}) \geq \gamma(f) d^{n-1} + 1$. Suppose that $e < n < \infty$, and write
\[
f^{(e)}(X) = X^S \phi(X), \quad S \geq 1, \: \phi(0) \neq 0.
\]
We again note that any zeros of $f^{(n)}$ which are common with a previous iterate belong to $f^{(n-e)}$ by Lemma \ref{COMROOTS} (i). Now,
\[
f^{(n)}(X)=f^{(e)} \left(f^{(n-e)}(X) \right) = \left(f^{(n-e)}(X) \right)^S \phi \left( f^{(n-e)}(X) \right).
\]
If $e>1$, then $\z(f^{(e)}) \geq d^{e-1}+1 > 2$, and otherwise $\z(f^{(e)}) > 2$ by assumption. Hence $\z(\phi) > 1$, and so by Lemma \ref{COMPOSZERO}, $Z(n) \geq \z \left( \phi \left( f^{(n-e)} \right) \right) \geq \gamma(\phi)d^{n-e}+1 \geq d^{n-e}+1$. \QEDB

\section{Proof of Theorem~\ref{ORD}} \label{sec:ord}

If $f \in \mathbb{F}[X]$, this is the main result of \cite{Popovych}, so assume otherwise, in which case we define $s = \lfloor (n-1)/2 \rfloor$. Recall the following lower bound from Lambe \cite{Lambe}, on the number of solutions to a linear Diophantine inequality:
 
\begin{lemma} \label{DIO}
Suppose that $m$ and $x_0, \ldots ,x_{r-1}$ are positive integers such that $\gcd(x_0, \ldots ,x_{r-1})=1$. Then the number of non-negative integer solutions $a_0, \ldots ,a_{r-1}$ to the inequality 
\[ \sum_{i=0}^{r-1} a_ix_i \leq m, \] 
is at least 
\[ \begin{pmatrix} m+r \\ r \end{pmatrix} \prod_{i=0}^{r-1} \frac{1}{x_i}, \] 
with equality when $x_0=...=x_{r-1}=1$.
\end{lemma}

Now, set $m = \bar{n}$. Since $\alpha$ is a root of $X^{m}h(X)-g(X)$, we have $\alpha^{m} = f(\alpha)$. As $m$ is a power of $q$, applying the Frobenius automorphism iteratively gives
\begin{equation} \label{eqn8}
\alpha^{m^i}=f^{(i)}(\alpha), \: i \geq 0.
\end{equation}
Consider the set 
\[ S = \left\lbrace \sum_{i=0}^{t-1}a_im^i : \sum_{i=0}^{t-1} a_id^i \leq s \right\rbrace. \] 
We will show that the powers $\alpha^a$, with $a \in S$, are distinct in $\mathbb{F}_{q^n}$, so from Lemma~\ref{DIO}, $\alpha$ has order at least 
\[ \# S \geq \begin{pmatrix} s+t \\ t  \end{pmatrix} \prod_{i=0}^{t-1}\frac{1}{d^i}. \] 
Suppose that there exist integers $a, b$ in $S$ such that $\alpha^a=\alpha^b$. Writing $a = \sum_{i=0}^{t-1} a_im^i$ and $b = \sum_{i=0}^{t-1} b_im^i$, we have 
\[ \prod_{i=0}^{t-1} \left( \alpha^{m^i} \right)^{a_i} = \prod_{i=0}^{t-1} \left( \alpha^{m^i} \right)^{b_i}. \] 
The equation \eqref{eqn8} then gives 
\[ \prod_{i=0}^{t-1} \left( f^{(i)}(\alpha) \right)^{a_i} = \prod_{i=0}^{t-1} \left( f^{(i)}(\alpha) \right)^{b_i}. \] 
Let 
\[ k_1(X) = \prod_{a_i>b_i} g_i(X)^{a_i-b_i} \prod_{a_i<b_i} h_i(X)^{b_i-a_i} \] 
and 
\[ k_2(X) = \prod_{a_i<b_i} g_i(X)^{b_i-a_i} \prod_{a_i>b_i} h_i(X)^{a_i-b_i}. \] 
Then $k_1(\alpha) = k_2(\alpha)$. Since $\alpha$ has degree $n$ and $k_1$ and $k_2$ have degree at most 
\[ \sum_{i=0}^{t-1} \max \left\lbrace a_i,b_i \right\rbrace d^i \leq 2s \leq n-1, \] 
we have $k_1(X) = k_2(X)$. Thus $\prod_{i=0}^{t-1} \left( f^{(i)}(X) \right)^{a_i-b_i} = 1$. Then $a_i-b_i=0$ for each $i$ by Corollary~\ref{IND}, and hence $a = b$. \QEDB \medskip

In light of Theorem~\ref{ORD}, we wish to determine whether such a pair $(g,h)$ of suitable polynomials always exists for all $n$. If this is so, we can construct a reliable algorithm for finding elements of high order in $\mathbb{F}_{q^n}$. Namely, checking $X^{\bar n}h(X)-g(X)$ for irreducible factors of degree $n$, for each appropriate pair $(g,h) \in \mathbb{F}_q[X]^2$. The case where $h(X)=1$ is considered in \cite{Gao}, where it is reasonably conjectured, but not proved, that for every $n$, there exists $g \in \mathbb{F}_q[X]$ with $\deg g \leq 2 \log_q n$, such that $X^{\bar n}-g(X)$ has an irreducible factor of degree $n$. \par 
For our more general situation, we make the following weaker conjecture,

\begin{conjecture} \label{CONJ}
Suppose $n \geq 1$, and let $T$ be the set of pairs $(g,h) \in \mathbb{F}_q[X]^2$ of degree not exceeding $d:=\left\lceil 2\log_qn \right\rceil$ such that $f=g/h$ satisfies the conditions from Corollary~\ref{IND}. Then there exists $(g,h) \in T$ such that $X^{\bar{n}}h(X)-g(X)$ has an irreducible factor of degree $n$.
\end{conjecture}

To give some evidence for this conjecture, we first obtain a rough lower bound for the order of $T$. See \cite{Benjamin} for the next lemma, regarding the probability that two polynomials in $\mathbb{F}_q[X]$ are relatively prime.

\begin{lemma} \label{PROBCOP}
Let $g$ and $h$ be randomly chosen from the set of polynomials in $\mathbb{F}_q[X]$ of degree $a$ and $b$ respectively, where $a$ and $b$ are not both zero. Then the probability that $g$ and $h$ are relatively prime is $1-1/q$.
\end{lemma}

Clearly, every pair $(g,h) \in \mathbb{F}_q[X]^2$ with $\deg g = d$, $\deg h = d-1$ and $\gcd(g,h)=1$ is an element of $T$. Thus, Lemma~\ref{PROBCOP}. gives
\begin{align}
\# T & \geq \left(1-\frac{1}{q} \right) \cdot (q-1)q^d \cdot (q-1)q^{d-1} \notag \\
& \geq \frac{(q-1)^3}{q^2} q^{4 \log_q n } = \frac{(q-1)^3}{q^2} n^4. \label{eqn9}
\end{align}

Now, consider the following result from \cite{Gao}:

\begin{lemma} \label{PROBIRR}
Let $P_q(m,n)$ be the probability of a random polynomial in $\mathbb{F}_q[X]$ of degree $m \geq n$ having at least one irreducible factor of degree $n$. Then
\[ P_q(m,n) \sim \frac{1}{n}, \quad \text{as } \: \: n \to \infty, \] uniformly for $q$ and $m \geq n$.
\end{lemma}

If we model $X^{\bar n}h(X)-g(X)$ as a random polynomial in $\mathbb{F}_q[X]$ for each $(g,h) \in T$, Lemma~\ref{PROBIRR}, in conjunction with \eqref{eqn9}, suggests that for large $n$, we expect on the order of $n^3$ pairs $(g,h) \in T$ such that $X^{\bar n}h(X)-g(X)$ has an irreducible factor of degree $n$. Thus it is plausible that at least one such pair exists.

\section{Proof of Theorem~\ref{SHIFT}} \label{sec:shift}

We now restrict the field $\mathbb{F}$ to having characteristic 0. The key tool of this section is Theorem \ref{ABC}, and so the results could perhaps be extended to characteristic $p$, given stronger conditions to ensure that one of the polynomials $A$, $B$ or $C$, to which we apply the theorem, has non-vanishing derivative. 

We now prove Theorem~\ref{SHIFT}. Suppose $F_1(X,u(X)),\ldots,F_n(X,u(X))$ are multiplicatively dependent, and and assume that no proper subset of these is also multiplicatively dependent, as we can remove functions until this is the case. Then every zero and pole of $F_i$ for $1 \leq i \leq n$ must be a zero or pole of $F_j$ for some $j \neq i$. This is because otherwise we would require $k_i = 0$ in the equation
\begin{equation} \label{eqn10}
\prod_{\ell = 1}^n F_{\ell}(X,u(X))^{k_{\ell}}=1,
\end{equation}
and hence the proper subset $\{ F_{\ell}(X,u(X) : 1 \leq \ell \leq n, \: \ell \neq i \}$ would be multiplicatively dependent. Hence, if $\alpha$ is a zero or pole or $F_i(X,u(X))$, there exists $j \neq i$ such that $F_i(\alpha,Y)$ and $F_j(\alpha,Y)$ have the common zero or pole $u(\alpha)$, giving $R_{ij}(\alpha)=0$. Thus, any zero or pole of $F_i(X,u(X))$ for $1 \leq i \leq n$ is a zero of $\prod_{1 \leq i < j} \prod_{i < j \leq n} R_{ij}$. In particular, since for all $i \neq j$, $R_{ij}$ is not identically zero, we have
\begin{equation} \label{eqn11}
\deg \rad \prod_{i=1}^n G_i(X,u(X))H_i(X,u(X)) \leq \sum_{1 \leq i < j} \sum_{i < j \leq n} \deg R_{ij} = E.
\end{equation}

Now, for $1 \leq i \leq n$, write
\[ F_i(X,Y) = \frac{G_i(X,Y)}{H_i(X,Y)} = \frac{\sum_{\nu=0}^{e_i}g_{i,\nu}(X)Y^{\nu}}{\sum_{\nu=0}^{e_i}h_{i,\nu}(X)Y^{\nu}}, \]
and assume, without loss of generality, that $g_{i,e_i}$ is not identically zero (if it is, we can replace $G_i$ with $H_i$, and $g_{i,e_i}$ with $h_{i,e_i}$ in the following definitions). For $1 \leq i < j \leq n$, define
\[ P(X) = g_{i,e_i}(X)G_j(X,u(X)), \quad Q(X) = g_{j,e_j}(X)u(X)^{e_j-e_i}G_i(X,u(X)), \]
and $D_{ij}(X) = \gcd(P(X),Q(X))$. Then set
\[ A(X) = \frac{P(X)}{D_{ij}(X)}, \quad B(X) = -\frac{Q(X)}{D_{ij}(X)}, \quad C(X) = -(A(X)+B(X)). \]
Then $A,B$, and $C$ are relatively prime polynomials with $A+B+C=0$. Suppose $\deg u \geq d_n$. Then 
\begin{equation} \label{eqn12}
\deg A = \deg P - \deg D_{ij} = \deg g_{i,e_i} + \deg g_{j,e_j}+e_j \deg u - \deg D_{ij},
\end{equation} 
which is positive as $R_{ij} \not\equiv 0$ ensures that $P \nmid Q$ and so $\deg D_{ij} < \deg P$. Thus $A$ has non-vanishing derivative. Moreover, in $C$, the term in $u(X)^{e_j}$ cancels out, giving 
\begin{equation} \label{eqn13}
\begin{split}
\deg C & \leq (e_j-1) \deg u \\
& + \max \{ \deg g_{i,e_i}+ \deg g_{j,e_j-1}, \deg g_{j,e_j}+\deg g_{i,e_i-1} \} - \deg D_{ij}.
\end{split}
\end{equation}
Therefore, we have by Lemma~\ref{ABC} and \eqref{eqn12},
\begin{align*}
\deg A & = \deg g_{i,e_i} + \deg g_{j,e_j}+e_j \deg u - \deg D_{ij} \\
& \leq \max \{ \deg A, \deg B, \deg C \} \\
& \leq \deg \rad ABC - 1 \\
& \leq \deg \rad G_iG_j + \deg g_{i,e_i} + \deg g_{j,e_j} + \deg C - 1.
\end{align*}
Then, \eqref{eqn11} and \eqref{eqn13} give
\begin{multline*}
e_j \deg u - \deg D_{ij} \leq E + (e_j-1) \deg u + \\ \max \{ \deg g_{i,e_i}+ \deg g_{j,e_j-1}, \deg g_{j,e_j}+\deg g_{i,e_i-1} \} - \deg D_{ij}
\end{multline*}
and hence,
\begin{align*}
\deg u & \leq E + \max \{ \deg g_{i,e_i}+ \deg g_{j,e_j-1}, \deg g_{j,e_j}+\deg g_{i,e_i-1} \} - 1 \\
& \leq E + 2d_n - 1. 
\end{align*}
Therefore, for $1 \leq i \leq n$, $G_i(X,u(X))$ is a product of at most $E$ distinct irreducible factors, with degree not exceeding $e_n(E+2d_n-1)+d_n$. If $w_0, \ldots ,w_{E-1}$ are the respective multiplicities of said factors, then up to multiplication by a non-zero constant, the number of possibilities for $G_i(X,u(X))$ is at most the number of non-negative integer solutions to the inequality
\[ \sum_{j=0}^{E-1} w_j \leq e_n(E+2d_n-1)+d_n, \]
which is at most $\binom{e_n(E+2d_n-1)+E+d_n}{E}$ from Lemma~\ref{DIO}. For each such possibility, say
\[ G_i(X,u(X)) = \sum_{j=0}^{d_i} \sum_{k=0}^{e_i} a_{jk} X^j u(X)^k = A \prod_{\ell=0}^{E-1} (X-\alpha_{\ell})^{b_\ell}, \]
if $u$ is monic then $A$ is uniquely determined. Moreover, we have
\[ u(X) \mid A \prod_{\ell=0}^{E-1} (X-\alpha_{\ell})^{b_\ell} - \sum_{j=0}^{d_i} a_{j0} X^j, \]
so there are finitely many possibilities for monic $u$.

For corollary \ref{POLYSHIFT}, we have $F_i(X,Y)=G_i(X,Y)=f_i(X)+Y$, giving, $R_{ij}(X)=f_j(X)-f_i(X)$ and $\deg R_{ij} \leq d_n$. Therefore $E \leq \frac{n(n-1)}{2}d_n=C$. Noting that $e_n=1$ in this case, up to constants there are at most $\binom{2C+3d_n-1}{C}$ possibilities for $f_i(X)+u(X)$, and hence for $u$. This completes the proof. \QEDB \medskip

\section{Comments}

Considering the case $\nu < \mu$ (which encompasses the polynomial case) of Theorem~\ref{MAIN}, and additionally Theorem~\ref{ZEROS}, it is of interest to obtain upper bounds for the value $e$ when it is finite. That is, bounds for the period of $0$ under iteration of a polynomial or rational function $f$. This problem is investigated in various contexts in \cite{Canci, Koch, Nark1, Nark2, Nark3, Pezda1}. Bounds on the values of the values of $\epsilon$, $\mu$ and $\nu$ in the rational function case are similarly of interest.

Another problem is to generalise Theorem \ref{ZEROS} to rational functions. Our approach used for the polynomial case can plausibly be extended to the situation where $\nu \leq \mu$, mirroring the proof of the relevant case in Theorem \ref{MAIN}, but applying an appropriate version of the main theorem in \cite{FucPet}. Such an extension, however, is not immediate for the case $\mu < \nu$.

Also, note that in the case $\mathbb{F}=\mathbb{C}$, Theorem~\ref{SHIFT} may be able to be generalised to several variables, where $F_i \in \mathbb{C}(X_1,\ldots,X_m,Y)$ and $u \in \mathbb{C}[X_1,\ldots,X_m]$, using an appropriate analogue of Mason's theorem (for example \cite[Theorem~2]{Bayat}).

\section*{Acknowledgement}

The author is grateful to Alina Ostafe and Igor Shparlinski for their ideas, comments and encouragement. He would also like to thank the referee for a careful reading and valuable suggestions.

\Address

\end{document}